\documentclass[12pt,a4paper]{article}

\usepackage[verbose]{geometry}
\usepackage[T1]{fontenc} 
\usepackage[utf8]{inputenc}

\ifdefined\ifdraftdoc

\else
  \usepackage{ifdraft}
  \makeatletter
  \let\ifdraftdoc\if@draft
  \makeatother
\fi

\usepackage{needspace}

\usepackage{multirow}
\usepackage[leqno]{amsmath}
\allowdisplaybreaks[4]
\usepackage{amsthm}
\usepackage{amssymb}
\usepackage{amsfonts}

\renewcommand{\qedsymbol}{$\blacksquare$}
\newcommand{\openqedsymbol}{$\square$}
\newcommand{\exercisesymbol}{$\lozenge$}
\newcommand{\obvioussymbol}{$\blacklozenge$}
\newcommand{\openproof}{\renewcommand{\qedsymbol}{\openqedsymbol}}

\newcommand{\obviousproof}{\renewcommand{\qedsymbol}{\obvioussymbol}}

\usepackage{mathrsfs}
\DeclareMathAlphabet{\mathcal}{OMS}{cmsy}{m}{n}
\usepackage{stmaryrd}

\usepackage{mathtools}

\usepackage[british]{babel}

\usepackage[autostyle]{csquotes}
\usepackage{hyphenat}

\usepackage{etoolbox}
\usepackage{xstring}

\usepackage[shortlabels]{enumitem}

\usepackage[dvipsnames,hyperref]{xcolor}

\usepackage{tikz}
\usepackage{tikz-cd}

\usetikzlibrary{arrows, shapes.geometric}
\tikzset{baseline=(current bounding box.center)}

\let\LiningNumbers\relax

\DeclareTextFontCommand{\textrmup}{\rmfamily\mdseries\upshape}
\DeclareTextFontCommand{\textbfup}{\rmfamily\bfseries\upshape}
\DeclareTextFontCommand{\textttup}{\ttfamily\mdseries\upshape}
\DeclareTextFontCommand{\textsfup}{\sffamily\mdseries\upshape}

\setlist[enumerate]{font=\upshape}

\DeclareFontFamily{OT1}{pzc}{}
\DeclareFontShape{OT1}{pzc}{mb}{it}{
   <-> s * [1.150] pzcmi7t
}{}

\DeclareMathAlphabet{\mathpzc}{OT1}{pzc}{mb}{it}

\DeclareMathSymbol{\epsilon}{\mathord}{letters}{"22}
\DeclareMathSymbol{\phi}{\mathord}{letters}{"27}
\DeclareMathSymbol{\varepsilon}{\mathord}{letters}{"0F}
\DeclareMathSymbol{\varphi}{\mathord}{letters}{"1E}

\usepackage{microtype}

\makeatletter
\newcommand*{\noprelistbreak}{\@nobreaktrue\nopagebreak}
\makeatother

\usepackage{xspace}

\newcommand{\hairspace}{\ifmmode\mskip1mu\else\kern0.08em\fi}
\newcommand{\mmhfill}{\hskip \textwidth minus \textwidth}

\makeatletter
\def\abbrdot{\@ifnextchar.{}{.\@\xspace}}
\makeatother
\newcommand{\etc}{etc\abbrdot}

\newcommand{\ie}{i.e\abbrdot}
\newcommand{\resp}{resp\abbrdot}

\newcommand{\opcit}{op.\@ cit\abbrdot}

\newcommand{\Chap}{Ch.\@~}

\newcommand{\Sect}{\S\hairspace}

\newcommand{\strong}[1]{\textbfup{#1}}

\renewcommand{\implies}{\ifmmode\Longrightarrow\else$\Rightarrow$\expandafter\xspace\fi}
\renewcommand{\iff}{\ifmmode\Longleftrightarrow\else$\Leftrightarrow$\expandafter\xspace\fi}

\makeatletter
\let\internal@prime\prime
\renewcommand{\prime}{\ifmmode\internal@prime\else$\sp{\internal@prime}$\expandafter\xspace\fi}
\makeatother

\ifdefined\nequiv

\else
  \def\nequiv{\not\equiv}
\fi

\newcommand{\blank}{\mathord{-}}
\newcommand{\pblank}{\mathord{(-)}}

\newcommand{\embedinto}{\hookrightarrow}
\newcommand{\hoto}{\Rightarrow}

\DeclareMathSymbol{.}{\mathpunct}{letters}{"3A}
\DeclareMathSymbol{:}{\mathrel}{operators}{"3A}

\DeclareMathOperator{\ob}{ob}
\DeclareMathOperator{\mor}{mor}

\DeclareMathOperator{\discr}{disc}

\newcommand{\Ho}[1][]{\mathop{\mathrm{Ho}\ifstrempty{#1}{}{\sb{#1}}}}

\newcommand{\LH}{\mathbf{L}^\mathrm{H}}

\newcommand{\id}{\mathrm{id}}

\AtBeginDocument{}

\newcommand*{\lr}[3]{\mathopen{}\mathclose{\left#1{#2}\right#3}}
\newcommand*{\lmr}[5]{\mathopen{}\mathclose{\left#1{#2}\,\middle#3\,{#4}\right#5}}

\newcommand*{\ordlr}[3]{\mathord{\left#1{#2}\right#3}}
\newcommand*{\argp}[1]{\ifstrempty{#1}{}{\lr({#1})}}
\newcommand*{\argptwo}[2]{\lr({#1, #2})}
\newcommand*{\argptwotwo}[4]{\lr({#1, #2; #3, #4})}
\newcommand*{\argb}[1]{\ifstrempty{#1}{}{\lr[{#1}]}}
\newcommand*{\parens}[1]{\lr({#1})}

\newcommand*{\bracket}[1]{\lr[{#1}]}
\newcommand*{\tuple}[1]{\ordlr({#1})}

\newcommand*{\setbuilder}[2]{\mathord{\lmr\{{#1}|{#2}\}}}
\newcommand*{\seqbuilder}[2]{\mathord{\lmr({#1}|{#2})}}

\ifdefined\fracslash

\else
  \def\fracslash{\mathbin{/}}
\fi
\ifdefined\divslash

\else
  \def\divslash{\mathbin{/}}
\fi

\newcommand*{\ul}[1]{\smash{\underline{#1}}\vphantom{#1}}

\makeatletter
\newcommand*{\optargp}{\new@ifnextchar\bgroup{\argp}{}}
\newcommand*{\optargptwo}{\new@ifnextchar\bgroup{\expandafter\optargptwo@first}{}}
\newcommand*{\optargptwo@first}[1]{\new@ifnextchar\bgroup{\argptwo{#1}}{\argp{#1}}}
\newcommand*{\optargptwotwo}{\new@ifnextchar\bgroup{\expandafter\optargptwotwo@first}{}}
\newcommand*{\optargptwotwo@first}[2]{\new@ifnextchar\bgroup{\argptwotwo{#1}{#2}}{\argptwo{#1}{#2}}}

\newcommand*{\sbtwo}[2]{\sb{\lr({{#1}, {#2}})}}

\newcommand*{\optsb}{\new@ifnextchar\bgroup{\sb}{}}
\newcommand*{\optsp}{\new@ifnextchar\bgroup{\sp}{}}
\newcommand*{\optsbtwo}{\new@ifnextchar\bgroup{\expandafter\optsbtwo@first}{}}
\newcommand*{\optsbtwo@first}[1]{\new@ifnextchar\bgroup{\sbtwo{#1}}{\sb{#1}}}

\newcommand*{\optsbargptwo}[1][]{\sb{#1}\optargptwo}

\newcommand*{\optspsbargptwo}[1][]{\sp{#1}\optsbargptwo}

\newcommand*{\set}[1]{\new@ifnextchar\bgroup{\setbuilder{#1}}{\ordlr\{{#1}\}}}
\newcommand*{\seq}[1]{\new@ifnextchar\bgroup{\seqbuilder{#1}}{\tuple{#1}}}
\makeatother

\ifdefined\hash

\else
  \def\hash{\#}
\fi

\ifdefined\powerset
  \undef\powerset
\else

\fi
\newcommand*{\powerset}[1][]{\mathscr{P}\ifstrempty{#1}{\hairspace}{\sb{#1}} \optargp}

\newcommand{\mbfs}{\mathbf{s}}
\newcommand{\mbfc}{\mathbf{c}}

\newcommand*{\smashsp}[2]{{#1}\sp{\smash{#2}}}

\newcommand*{\op}[1]{\smashsp{#1}{\mathrm{op}}}

\newcommand*{\inv}[2][1]{{#2}^{-#1}}

\makeatletter
\newcommand*{\ctchoice}[2]{%
  \ifdef{\ct@homstyle}%
  {#2}%
  {#1}}
\newcommand*{\ctchoicetwo}[5]{%
  \ifdef{\ct@homcatstyle}%
  {#2}{%
  \ifdef{\ct@transfstyle}%
  {#3}{%
  \ifdef{\ct@catstyle}%
  {#4}{%
  \ifdef{\ct@homstyle}%
  {#5}{%
  {#1}}}}}}
  
\newcommand*{\DefineCategory}[2]{\DeclareRobustCommand{#1}{\ctchoice{\mathbf{#2}}{\mathbf{#2}}}}

\newcommand*{\cat}[1]{%
  \begingroup%
  \def\ct@catstyle{cat}%
  #1%
  \endgroup}
\newcommand*{\bicat}[1]{%
  \begingroup
  \def\ct@bicatstyle{bicat}%
  #1%
  \endgroup}
\newcommand*{\Hom}[1][]{%
  \begingroup%
  \def\ct@homstyle{hom}%
  \ifstrempty{#1}{\mathrm{Hom}}{#1}%
  \endgroup%
  \optargptwo}
\newcommand*{\HHom}[1][]{%
  \begingroup%
  \def\ct@homcatstyle{homcat}%
  \ifstrempty{#1}{\mathbf{Hom}}{#1}%
  \endgroup
  \optargptwo}
\newcommand*{\hpty}[1][]{%
  \begingroup%
  \def\ct@transfstyle{transf}%
  \ifstrempty{#1}{\mathrm{Nat}}{#1}%
  \endgroup
  \optargptwotwo}

\makeatother

\DefineCategory{\Set}{Set}

\DefineCategory{\Simplex}{\Delta}
\DefineCategory{\SSet}{sSet}
\DefineCategory{\SSSet}{ssSet}

\newcommand{\Cat}{\ctchoicetwo{\mathfrak{Cat}}{\mathbf{Fun}}{\mathbf{Fun}}{\mathbf{Cat}}{\mathrm{Fun}}}

\newcommand*{\commacat}[2]{\ordlr({#1 \mathbin{\downarrow} #2})}

\newcommand*{\overcat}[2]{\mathord{#1 \sb{\mathord{\divslash} #2}}}
\newcommand*{\undercat}[2]{\mathord{\sp{#1 \mathord{\divslash}} #2}}

\newcommand*{\nv}[1][]{\mathrm{N}\sp{#1}\optargp}
\newcommand*{\rn}[2][]{\ordlr|{#2}|\ifstrempty{#1}{}{\sb{#1}}}
\DeclareMathOperator{\Tot}{Tot}

\newcommand{\skel}[1][]{\mathrm{sk}\sb{#1}\optargp}

\newcommand{\ccoskel}[1][]{\mathrm{cosk}\sp{#1}\optargp}

\newcommand{\totalL}{\mathbf{L}}
\newcommand{\totalR}{\mathbf{R}}

\newcommand*{\homs}{\mathpzc{h}\optsb}
\newcommand*{\ulhoms}{\ul{\mathpzc{h}}\optsb}
\newcommand*{\hathoms}{\hat{\mathpzc{h}}\optsb}

\newcommand*{\sHom}[1][]{\ifstrempty{#1}{\smash{\mathpzc{Hom}}\vphantom{\mathrm{Hom}}}{\smash{\mathpzc{Hom}}\vphantom{\mathrm{Hom}}_{#1}}\optargptwo}
\newcommand*{\RHom}[1][]{\ifstrempty{#1}{\totalR \mathrm{Hom}}{\totalR \mathrm{Hom}_{#1}}\optargptwo}

\newcommand*{\xHom}[3][]{\ifstrempty{#1}{\ordlr[{{#2}, {#3}}]}{\ordlr[{{#2}, {#3}}]\sb{#1}}}
\newcommand*{\ulHom}[1][]{\ifstrempty{#1}{\ul{\Hom}}{\ul{\Hom[#1]}}\optargptwo}

\newcommand*{\Ext}{\mathrm{Ext}\optspsbargptwo}

\newcommand*{\Func}[2]{\xHom{#1}{#2}}

\newcommand*{\wcolim}[2][]{\mathchoice%
  {\mathop{{\varinjlim}^{#2}}_{{#1}\phantom{#2}}}%
  {\varinjlim_{#1}^{#2}}%
  {\varinjlim_{#1}^{#2}}%
  {\varinjlim_{#1}^{#2}}}

\newcommand{\BKcolim}[1][]{\wcolim[#1]{\mathrm{BK}}}
\newcommand{\KBcolim}[1][]{\wcolim[#1]{\mathrm{KB}}}
\newcommand{\laxcolim}[1][]{\wcolim[#1]{\mathrm{Th}}}
\newcommand{\oplaxcolim}[1][]{\wcolim[#1]{\mathrm{Gr}}}

\makeatletter
\renewcommand\footnoterule{%
  \vfill
  \kern-3\p@
  \hrule\@width.4\columnwidth
  \kern2.6\p@}

\renewcommand*{\@makefnmark}{\hbox{\textsuperscript{\normalfont [\LiningNumbers\@thefnmark]}}}
\renewcommand*{\@makefntext}[1]{\parindent 1.0em\noindent\ifdefempty{\@thefnmark}{}{\hb@xt@-1.0em{\hss \normalfont [\LiningNumbers \@thefnmark]}\hspace{1.0em}}#1}
\@addtoreset{footnote}{chapter}
\newcommand{\footpar}[1]{\gdef\@thefnmark{}\@footnotetext{#1}}

\newcommand{\hangsecnum}{\def\@seccntformat##1{\llap{\csname the##1\endcsname\quad}}}
\makeatother

\usepackage[natbib,safeinputenc,backend=biber,style=authoryear-comp,bibstyle=authoryear]{biblatex}

\DeclareSortingScheme{mysort}{
  \sort{
    \field{presort}
  }
  \sort[final]{
    \field{sortkey}
  }
  \sort{
    \name{sortname}
    \name{author}
    \name{editor}
    \name{translator}
    \field{sorttitle}
    \field{title}
  }
  \sort{
    \field{sortyear}
    \field{year}
  }
  \sort{
    \field{month}
  }
  \sort{
    \field{day}
  }
  \sort{
    \field{journal}
  }
  \sort{
    \field[padside=left,padwidth=4,padchar=0]{volume}
    \literal{0000}
  }
  \sort{
    \field{series}
  }
  \sort{
    \field[padside=left,padwidth=4,padchar=0]{number}
    \literal{0000}
  }
  \sort{
    \field{sorttitle}
    \field{title}
  }
}

\usepackage[colorlinks,linkcolor=Blue,citecolor=Blue,urlcolor=blue,pdfpagelabels,pdffitwindow=false]{hyperref}

\usepackage{thmtools}

\declaretheorem[style=plain,parent=section,title=Theorem,refname=theorem,Refname=Theorem]{thm}
\declaretheorem[style=plain,sibling=thm,title=Proposition,refname=proposition,Refname=Proposition]{prop}
\declaretheorem[style=plain,sibling=thm,title=Lemma,refname=lemma,Refname=Lemma]{lem}
\declaretheorem[style=plain,sibling=thm,title=Corollary,refname=corollary,Refname=Corollary]{cor}

\declaretheorem[style=definition,sibling=thm,title=Definition,refname=definition,Refname=Definition]{dfn}

\declaretheoremstyle[style=remark,headfont=\scshape]{remark}
\declaretheorem[style=remark,sibling=thm,title=Remark,refname=remark,Refname=Remark]{remark}

\declaretheorem[style=plain,numbered=no,title=Theorem]{thm*}
\declaretheorem[style=plain,numbered=no,title=Proposition]{prop*}
\declaretheorem[style=remark,numbered=no,title=Remark]{remark*}
\declaretheorem[style=definition,numbered=no,title=Example]{example*}

\makeatletter
\@counteralias{numpar}{thm}

\makeatother

\AtBeginDocument{%
\hyphenation{pull-back pull-backs push-out push-outs co-pro-duct co-pro-ducts co-com-plete sub-ob-ject sub-ob-jects more-over quasi-com-pact de-gen-er-ate pro-vided Gro-then-dieck Law-vere mod-ule sub-al-ge-bra fi-brant co-fi-brant fi-brant-ly co-fi-brant-ly bi-cat-e-go-ry bi-cat-e-gor-ic-al bi-cat-e-gor-ic-al-ly qua-si-cat-e-go-ry qua-si-cat-e-go-ries qua-si-cat-e-go-ric-al pre-sheaf pre-sheaves}%
}

\renewcommand*{\bibnamedash}{}
\newcommand*{\authoryearpunct}{\hspace*{-\bibhang}\bibsentence}

\makeatletter

\appto\citesetup{\LiningNumbers}

\renewbibmacro*{cite:label}{%
  \iffieldundef{label}
    {\printtext[bibhyperref]{\printfield[citetitle]{labeltitle}}}
    {\printtext[bibhyperref]{\printfield{label}}}}

\renewbibmacro*{date+extrayear}{%
  \iffieldundef{shorthand}
    {\iffieldundef{labelyear}
      {}
      {\printtext[brackets]{\printfield{labelyear}\printfield{extrayear}}}\addspace}
    {\printtext[brackets]{\printfield{shorthand}}}}

\renewbibmacro*{author}{%
  \ifboolexpr{
    test \ifuseauthor
    and
    not test {\ifnameundef{author}}
  }
    {\usebibmacro{bbx:dashcheck}
       {\bibnamedash}
       {\usebibmacro{bbx:savehash}%
        \printnames{author}%
    \newline\setunit{\authoryearpunct}}%
     \iffieldundef{authortype}
       {}
       {\usebibmacro{authorstrg}%
    \setunit{\addspace}}}%
    {\global\undef\bbx@lasthash
     \usebibmacro{labeltitle}%
     \setunit*{\addspace}}%
  \usebibmacro{date+extrayear}}

\renewbibmacro*{bbx:editor}[1]{%
  \ifboolexpr{
    test \ifuseeditor
    and
    not test {\ifnameundef{editor}}
  }
    {\usebibmacro{bbx:dashcheck}
       {\bibnamedash}
       {\printnames{editor}%
    \newline\setunit{\authoryearpunct}%
    \usebibmacro{bbx:savehash}}%
     \usebibmacro{#1}%
     \clearname{editor}%
     \setunit{\addspace}}%
    {\global\undef\bbx@lasthash
     \usebibmacro{labeltitle}%
     \setunit*{\addspace}}%
  \usebibmacro{date+extrayear}}

\renewbibmacro*{bbx:translator}[1]{%
  \ifboolexpr{
    test \ifusetranslator
    and
    not test {\ifnameundef{translator}}
  }
    {\usebibmacro{bbx:dashcheck}
       {\bibnamedash}
       {\printnames{translator}%
    \newline\setunit{\authoryearpunct}%
    \usebibmacro{bbx:savehash}}%
     \usebibmacro{translator+othersstrg}%
     \clearname{translator}%
     \setunit{\addspace}}%
    {\global\undef\bbx@lasthash
     \usebibmacro{labeltitle}%
     \setunit*{\addspace}}%
  \usebibmacro{date+extrayear}}

\makeatother

\appto\bibsetup{\raggedright}
\appto\bibfont{\LiningNumbers\small}

\usepackage{titling}
\usepackage{fancyhdr}
\hangsecnum

\title{Revisiting function complexes and simplicial localisation}
\author{Zhen~Lin Low}
\date{29 September 2014}

\begin{document}

\maketitle
\footpar{Department of Pure Mathematics and Mathematical Statistics, University of Cambridge, Cambridge, UK. \textsc{E-mail address}: \texttt{Z.L.Low@dpmms.cam.ac.uk}}

\begin{abstract}
In this paper three results are established: firstly, that the homotopy function complexes of Dwyer and Kan can be defined as certain total right derived functors; secondly, that they \emph{functorially} compute the homotopy type of the hom-spaces in the simplicial localisation; and thirdly, that they can be computed by fibrant replacements in a suitable left Bousfield localisation of the projective model structure on simplicial presheaves.
\end{abstract}

\section*{Introduction}

A (closed) model category in the sense of \citet{Quillen:1967} is an abstraction of the homotopy theory of topological spaces: it is a category equipped with notions of `path space', `homotopy', \etc that behave much like their namesakes in the category of topological spaces. As such, one might have also expected a notion of `mapping space', but initially, these were only defined for simplicial model categories. The first general definition appeared in the work of \citet{Dwyer-Kan:1980a,Dwyer-Kan:1980b,Dwyer-Kan:1980c}: in fact, they introduced \emph{three} explicit models for mapping spaces and showed that they are all weakly homotopy equivalent. In brief: 
\begin{itemize}
\item The first model is constructed using the methods of homotopical algebra applied to the category of simplicially enriched categories over a fixed object set: one essentially takes a cofibrant resolution of the model category itself and then localises that.

\item The second model is built using ``reduced hammocks'' and resembles Yoneda's \parencite*{Yoneda:1954} construction of $\Ext[*]$-groups in terms of diagrams.

\item The third model is defined in terms of simplicial and cosimplicial resolutions, which is essentially the same as the construction of $\Ext[*]$-groups in terms of injective and projective resolutions.
\end{itemize}

It was shown in the second Dwyer--Kan paper that the first two models are weakly homotopy equivalent in a functorial way, and the main result of the third Dwyer--Kan paper was that the last two models are weakly homotopy equivalent, modulo a minor gap which was repaired by \citet[\Sect 7]{Mandell:1999} and \citet{Dugger:2006b} independently. Unfortunately, the complications so introduced make it non-obvious whether the weak homotopy equivalence constructed can be made functorial; one of the goals of this paper is to clarify this point by giving yet another proof of the Dwyer--Kan result.

We will revisit all three Dwyer--Kan constructions in this paper, following the outline below:
\begin{itemize}
\item In \Sect 1, we review the theory of homotopy colimits of diagrams of simplicial sets.

\item In \Sect 2, we extend the analogy with homological algebra indicated in the first paragraph by showing that homotopy function complexes can be defined as total right derived functors of certain functors defined on the category of (co)simplicial objects.

\item In \Sect 3, we show that homotopy function complexes are \emph{naturally} weakly homotopy equivalent to the hom-spaces of the hammock localisation.

\item In \Sect 4, we use left Bousfield localisation to show that representable presheaves admit a generalised right derived functor, which can be computed in terms of the hom-spaces of the (standard) simplicial localisation.
\end{itemize}

\subsection*{Conventions}

\begin{itemize}
\item We will mostly use the same notations and definitions as in \citep{Dwyer-Kan:1980a,Dwyer-Kan:1980b,Dwyer-Kan:1980c}.

\item We will also need the notions of `homotopical equivalence', `right approximation', and `deformable functor' from \citep{DHKS}. 

\item For simplicity, we will restrict our attention to \emph{small} model categories with \emph{functorial} factorisations.

\item However, to avoid triviality, we will only assume that our model categories have \emph{finite} limits and colimits.

\item We will use underlines to indicate simplicial enrichment.
\end{itemize}

The smallness hypothesis is easily circumvented under the assumption of a suitable universe axiom, but removing the functoriality hypothesis requires a small extension of the DHKS theory of deformable functors. The author intends to address this in future work.

\subsection*{Acknowledgements}

Long discussions with Aaron Mazel-Gee helped clarify many points in the original Dwyer--Kan papers. Thanks also to Denis-Charles Cisinski for pointing out that \autoref{thm:Grothendieck.homotopy.cofinality} could be found in the work of \citet{Maltsiniotis:2005a}.

The author gratefully acknowledges financial support from the Cambridge Commonwealth, European and International Trust and the Department of Pure Mathematics and Mathematical Statistics.

\section{Homotopy colimits}

We will need several explicit models for homotopy colimits of diagrams of simplicial sets. The following are based on the formulae of \citet[\Chap XII]{Bousfield-Kan:1972}:

\begin{dfn}
\needspace{2.5\baselineskip}
Let $X : \mathcal{C} \to \cat{\SSet}$ be a small diagram.
\begin{itemize}
\item The \strong{Bousfield--Kan colimit} of $X$ is the simplicial set $\BKcolim[\mathcal{C}] X$ defined by the formula below,
\[
\parens{\textstyle \BKcolim[\mathcal{C}] X}_n = \coprod_{\tuple{c_0, \ldots, c_n}} \Hom[\mathcal{C}]{c_{n-1}}{c_n} \times \cdots \times \Hom[\mathcal{C}]{c_0}{c_1} \times X \argp{c_0}_n
\]
where the disjoint union is indexed over $\parens{n + 1}$-tuples of objects in $\mathcal{C}$, with the following face and degeneracy operators:
\begin{align*}
d^n_0 \argp{f_n, \ldots, f_1, x} & = \tuple{f_n, \ldots, f_2, d^n_0 \argp{X \argp{f_1} \argp{x}}} \\
d^n_i \argp{f_n, \ldots, f_1, x} & = \tuple{f_n, \ldots, f_{i+1} \circ f_i, \ldots, f_1, d^n_i \argp{x}} \\
d^n_n \argp{f_n, \ldots, f_1, x} & = \tuple{f_{n-1}, \ldots, f_1, d^n_n \argp{x}} \\
s^n_0 \argp{f_n, \ldots, f_1, x} & = \tuple{f_n, \ldots, f_1, \id_{c_0}, s^n_0 \argp{x}} \\
s^n_i \argp{f_n, \ldots, f_1, x} & = \tuple{f_n, \ldots, f_{i+1}, \id_{c_i}, f_i, \ldots, f_1, s^n_i \argp{x}} \\
s^n_n \argp{f_n, \ldots, f_1, x} & = \tuple{\id_{c_n}, f_n, \ldots, f_1, s^n_n \argp{x}}
\end{align*}

\item The \strong{dual Bousfield--Kan colimit} of $X$ is the simplicial set $\KBcolim[\mathcal{C}] X$ defined by the formula below,
\[
\parens{\textstyle \KBcolim[\mathcal{C}] X}_n = \coprod_{\tuple{c_0, \ldots, c_n}} X \argp{c_n}_n \times \Hom[\mathcal{C}]{c_n}{c_{n-1}} \times \cdots \times \Hom[\mathcal{C}]{c_1}{c_0}
\]
where the disjoint union is indexed over $\parens{n + 1}$-tuples of objects in $\mathcal{C}$, with the following face and degeneracy operators:
\begin{align*}
d^n_0 \argp{x, f_n, \ldots, f_1} & = \tuple{d^n_0 \argp{x}, f_n, \ldots, f_2} \\
d^n_i \argp{x, f_n, \ldots, f_1} & = \tuple{d^n_i \argp{x}, f_n, \ldots, f_i \circ f_{i+1}, \ldots, f_1} \\
d^n_n \argp{x, f_n, \ldots, f_1} & = \tuple{d^n_n \argp{X \argp{f_n} \argp{x}}, f_{n-1}, \ldots, f_1} \\
s^n_0 \argp{x, f_n, \ldots, f_1} & = \tuple{s^n_0 \argp{x}, f_n, \ldots, f_1, \id_{c_0}} \\
s^n_i \argp{x, f_n, \ldots, f_1} & = \tuple{s^n_i \argp{x}, f_n, \ldots, f_{i+1}, \id_{c_i}, f_i, \ldots, f_1} \\
s^n_n \argp{x, f_n, \ldots, f_1} & = \tuple{s^n_n \argp{x}, \id_{c_n}, f_n, \ldots, f_1}
\end{align*}
\end{itemize}
\end{dfn}

\begin{remark}
We have a natural isomorphism relating the two constructions:
\begin{align*}
\textstyle \BKcolim[\mathcal{C}] \op{X} & \cong \op{\parens{\textstyle \KBcolim[\mathcal{C}] X}} \\
\tuple{f_n, \ldots, f_1, x} & \mapsto \tuple{x, f_1, \ldots, f_n}
\end{align*}
\end{remark}

\begin{remark*}
The above convention is chosen so that the following formula holds,
\[
\textstyle \BKcolim[\mathcal{C}] \Delta 1 \cong \nv{\mathcal{C}}
\]
where $\Delta 1$ is the constant diagram of shape $\mathcal{C}$ with value $1 \cong \Delta^0$ and $\nv{\mathcal{C}}$ is the nerve of $\mathcal{C}$. Since the ``underlying simplicial set'' of a category $\mathcal{C}$ is defined to be $\op{\nv{\mathcal{C}}}$ in \citep{Bousfield-Kan:1972}, the formula for homotopy colimits appearing in \opcit actually corresponds to what we call `dual Bousfield--Kan colimit'. The same is true for the formula appearing in \citep[\Chap 18]{Hirschhorn:2003}.
\end{remark*}

\begin{lem}
\label{lem:Bousfield-Kan.homotopical.properties}
Let $\phi : X \hoto Y$ be a natural transformation of small diagrams $\mathcal{C} \to \cat{\SSet}$.
\begin{itemize}
\item If the components $\phi_c : X \argp{c} \to Y \argp{c}$ are all weak homotopy equivalences, then the induced morphism $\BKcolim[\mathcal{C}] \phi : \BKcolim[\mathcal{C}] X \to \BKcolim[\mathcal{C}] Y$ is also a weak homotopy equivalence.

\item If the components $\phi_c : X \argp{c} \to Y \argp{c}$ are all weak homotopy equivalences, then the induced morphism $\KBcolim[\mathcal{C}] \phi : \KBcolim[\mathcal{C}] X \to \KBcolim[\mathcal{C}] Y$ is also a weak homotopy equivalence.
\end{itemize}
\end{lem}
\begin{proof} \openproof
See Lemma 4.2 in \citep[\Chap XII]{Bousfield-Kan:1972} or Theorem 18.5.1 in \citep{Hirschhorn:2003}.
\end{proof}

\begin{lem}[The Bousfield--Kan comparison]
\label{lem:diagonal.is.hocolim}
Let $X_{\bullet}$ be a bisimplicial set and let $\rn{X_{\bullet}}$ be the diagonal simplicial set: 
\[
\rn{X_{\bullet}}_n = \parens{X_n}_n
\]
\begin{itemize}
\item There is a natural weak homotopy equivalence $\BKcolim[\op{\cat{\Simplex}}] X_{\bullet} \to \rn{X_{\bullet}}$.

\item There is a natural weak homotopy equivalence $\KBcolim[\op{\cat{\Simplex}}] X_{\bullet} \to \rn{X_{\bullet}}$.
\end{itemize}
\end{lem}
\begin{proof} \openproof
See paragraph 4.3 in \citep[\Chap XII]{Bousfield-Kan:1972} or Theorem 18.7.4 in \citep{Hirschhorn:2003}.
\end{proof}

We will also need two versions of the Grothendieck construction:

\begin{dfn}
Let $\mathcal{X} : \mathcal{C} \to \cat{\Cat}$ be a small diagram.
\begin{itemize}
\item The \strong{lax colimit} for $\mathcal{X}$ is the category $\laxcolim[\mathcal{C}] \mathcal{X}$ defined below:
\begin{itemize}
\item The objects are pairs $\tuple{c, x}$ where $c$ is an object in $\mathcal{C}$ and $x$ is an object in $\mathcal{X} \argp{c}$.

\item The morphisms $\tuple{c', x'} \to \tuple{c, x}$ are pairs $\tuple{f, g}$ where $f : c' \to c$ is a morphism in $\mathcal{C}$ and $g : \mathcal{X} \argp{f} \argp{x'} \to x$ is a morphism in $\mathcal{X} \argp{c}$.

\item Composition and identities are inherited from $\mathcal{C}$ and $\mathcal{X}$.
\end{itemize}

\needspace{2.5\baselineskip}
\item The \strong{oplax colimit} for $\mathcal{X}$ is the category $\oplaxcolim[\mathcal{C}] \mathcal{X}$ defined below:
\begin{itemize}
\item The objects are pairs $\tuple{c, x}$ where $c$ is an object in $\mathcal{C}$ and $x$ is an object in $\mathcal{X} \argp{c}$.

\item The morphisms $\tuple{c', x'} \to \tuple{c, x}$ are pairs $\tuple{f, g}$ where $f : c \to c'$ is a morphism in $\mathcal{C}$ and $g : x' \to \mathcal{X} \argp{f} \argp{x}$ is a morphism in $\mathcal{X} \argp{c'}$.

\item Composition and identities are inherited from $\mathcal{C}$ and $\mathcal{X}$.
\end{itemize}
\end{itemize}
\end{dfn}

\begin{remark*}
It may help to observe that the canonical projection $\laxcolim[\mathcal{C}] \mathcal{X} \to \mathcal{C}$ is a Grothendieck opfibration, whereas the canonical projection $\oplaxcolim[\mathcal{C}] \mathcal{X} \to \op{\mathcal{C}}$ is a Grothendieck fibration. Thus $\oplaxcolim[\mathcal{C}] \mathcal{X}$ is the original Grothendieck construction.
\end{remark*}

\begin{remark*}
In the notation of \citet{Dwyer-Kan:1980b}, $\laxcolim[\mathcal{C}] \mathcal{X}$ is ${*} \otimes_\mathcal{C} \mathcal{X}$, and $\oplaxcolim[\mathcal{C}] \mathcal{X}$ is $\mathcal{X} \otimes_{\op{\mathcal{C}}} {*}$.
\end{remark*}

The notation $\laxcolim[\mathcal{C}] \mathcal{X}$ is in honour of the following result of \citet{Thomason:1979}:

\begin{thm}[Thomason's homotopy colimit theorem]
\label{thm:Thomason.hocolim}
Let $\mathcal{X} : \mathcal{C} \to \cat{\Cat}$ be a small diagram.
\begin{itemize}
\item There is a weak homotopy equivalence
\[
\textstyle \BKcolim[\mathcal{C}] \nv \circ \mathcal{X} \to \nv{\laxcolim[\mathcal{C}] \mathcal{X}}
\]
which is moreover natural in $\mathcal{C}$ and $\mathcal{X}$.

\item There is a weak homotopy equivalence
\[
\textstyle \KBcolim[\mathcal{C}] \nv \circ \mathcal{X} \to \nv{\oplaxcolim[\mathcal{C}] \mathcal{X}}
\]
which is moreover natural in $\mathcal{C}$ and $\mathcal{X}$.
\end{itemize}
\end{thm}

In addition, we need a homotopy cofinality theorem. Following \citep{Grothendieck:1983}:

\begin{dfn}
\ \noprelistbreak
\begin{itemize}
\item A \strong{left aspherical functor} is a functor $u : \mathcal{A} \to \mathcal{B}$ such that, for each object $b$ in $\mathcal{B}$, the nerve of the comma category $\commacat{b}{u}$ is a weakly contractible simplicial set.

\item A \strong{right aspherical functor} is a functor $u : \mathcal{A} \to \mathcal{B}$ such that, for each object $b$ in $\mathcal{B}$, the nerve of the comma category $\commacat{u}{b}$ is a weakly contractible simplicial set.
\end{itemize}
\end{dfn}

\begin{remark}
Since a simplicial set $X$ is weakly contractible if and only if $\op{X}$ is weakly contractible, a functor $u : \mathcal{A} \to \mathcal{B}$ is left aspherical if and only if $\op{u} : \op{\mathcal{A}} \to \op{\mathcal{B}}$ is right aspherical.
\end{remark}

The homotopical significance of these functors is hinted at by a result of \citet[\Sect 1]{Quillen:1973a}:

\begin{thm}[Quillen's Theorem A]
\label{thm:Quillen-A}
If $u : \mathcal{A} \to \mathcal{B}$ is either a left or right aspherical functor, then $\nv{u} : \nv{\mathcal{A}} \to \nv{\mathcal{B}}$ is a weak homotopy equivalence. 
\end{thm}

However, one can say more. The following result is originally due to \citet{Grothendieck:1991}.

\begin{thm}[Homotopy cofinality]
\label{thm:Grothendieck.homotopy.cofinality}
Let $u : \mathcal{A} \to \mathcal{B}$ be a functor between small categories and let $\mathcal{X} : \mathcal{B} \to \cat{\Cat}$ be a diagram.
\begin{enumerate}[(i)]
\item There are pullback diagrams in $\cat{\Cat}$ of the forms below,
\[
\mmhfill
\begin{tikzcd}
\laxcolim[\mathcal{A}] \mathcal{X} \circ u \dar \rar &
\laxcolim[\mathcal{B}] \mathcal{X} \dar \\
\mathcal{A} \rar[swap]{u} &
\mathcal{B}
\end{tikzcd}
\mmhfill
\begin{tikzcd}
\oplaxcolim[\mathcal{A}] \mathcal{X} \circ u \dar \rar &
\oplaxcolim[\mathcal{B}] \mathcal{X} \dar \\
\op{\mathcal{A}} \rar[swap]{\op{u}} &
\op{\mathcal{B}}
\end{tikzcd}
\mmhfill
\]
where the vertical arrows are the canonical projections, the top horizontal arrows are functorial in $\mathcal{X}$, and in the left (\resp right) diagram, the top horizontal arrow is an opcartesian (\resp cartesian) functor.

\item If $u : \mathcal{A} \to \mathcal{B}$ is left aspherical, then $\laxcolim[\mathcal{A}] \mathcal{X} \circ u \to \laxcolim[\mathcal{B}] \mathcal{X}$ is left aspherical and $\oplaxcolim[\mathcal{A}] \mathcal{X} \circ u \to \oplaxcolim[\mathcal{B}] \mathcal{X}$ is right aspherical.
\end{enumerate}
\end{thm}
\begin{proof}
(i). Straightforward.

\bigskip\noindent
(ii). The two halves of the claim are formally dual; the second version is a special case of Corollaire 4.16 in \citep{Maltsiniotis:2005a} (in view of Exemple 2.3 and Définition 4.6 in \opcit).
\end{proof}

\section{Derived hom-spaces}

Let $\mathcal{M}$ be a model category and let $\mbfs \mathcal{M}$ (\resp $\mbfc \mathcal{M}$) be the category of simplicial (\resp cosimplicial) objects in $\mathcal{M}$. As is well known,\footnote{See \citep[\Sect 5.2]{Hovey:1999} or \citep[\Chap 15]{Hirschhorn:2003}.} $\mbfs \mathcal{M}$ and $\mbfc \mathcal{M}$ have Reedy model structures, wherein the weak equivalences are the morphisms that are degreewise weak equivalences in $\mathcal{M}$.

\begin{prop}
\label{prop:simplicial.object.Quillen.adjunction}
\needspace{3\baselineskip}
Let $\pblank_0 : \mbfs \mathcal{M} \to \mathcal{M}$ be the functor that sends a simplicial object $B_{\bullet}$ in $\mathcal{M}$ to the component $B_0$. Then:
\begin{enumerate}[(i)]
\item $\pblank_0 : \mbfs \mathcal{M} \to \mathcal{M}$ has a left adjoint, namely the functor $\skel[0] : \mathcal{M} \to \mbfs \mathcal{M}$ that sends each object $A$ in $\mathcal{M}$ to the constant simplicial object with value $A$.

\item The adjunction
\[
\skel[0] \dashv \pblank_0 : \mbfs \mathcal{M} \to \mathcal{M}
\]
is a Quillen adjunction, and the unit is an isomorphism.

\item The induced functor $\Ho \skel[0] : \Ho \mathcal{M} \to \Ho \mbfs \mathcal{M}$ is fully faithful.
\end{enumerate}
Dually, let $\pblank^0 : \mbfc \mathcal{M} \to \mathcal{M}$ be the functor that sends a cosimplicial object $A^{\bullet}$ in $\mathcal{M}$ to the component $A^0$. Then:
\begin{enumerate}[(i\prime)]
\item $\pblank^0 : \mbfc \mathcal{M} \to \mathcal{M}$ has a right adjoint, namely the functor $\ccoskel[0] : \mathcal{M} \to \mbfc \mathcal{M}$ that sends each object $B$ in $\mathcal{M}$ to the constant cosimplicial object with value $B$.

\item The adjunction
\[
\pblank^0 \dashv \ccoskel[0] : \mathcal{M} \to \mbfc \mathcal{M}
\]
is a Quillen adjunction, and the counit is an isomorphism.

\item The induced functor $\Ho \ccoskel[0] : \Ho \mathcal{M} \to \Ho \mbfc \mathcal{M}$ is fully faithful.
\end{enumerate}
\end{prop}
\begin{proof} \obviousproof
Straightforward.
\end{proof}

\begin{dfn}
\ \noprelistbreak
\begin{itemize}
\item A \strong{weakly constant simplicial object} in $\mathcal{M}$ is a simplicial object $B_{\bullet}$ such that the counit $\skel[0]{B_0} \to B_{\bullet}$ is a Reedy weak equivalence in $\mbfs \mathcal{M}$. 

We write $\mbfs_\mathrm{w} \mathcal{M}$ for the full subcategory of $\mbfs \mathcal{M}$ spanned by the weakly constant simplicial objects.

\item A \strong{weakly constant cosimplicial object} in $\mathcal{M}$ is a cosimplicial object $A^{\bullet}$ such that the unit $A^{\bullet} \to \ccoskel[0]{A^0}$ is a Reedy weak equivalence in $\mbfc \mathcal{M}$.

We write $\mbfc_\mathrm{w} \mathcal{M}$ for the full subcategory of $\mbfc \mathcal{M}$ spanned by the weakly constant cosimplicial objects.

\item A \strong{simplicial resolution} in $\mathcal{M}$ is an weakly constant simplicial object in $\mathcal{M}$ that is also Reedy-fibrant in $\mbfs \mathcal{M}$.

We write $\mbfs_\mathrm{r} \mathcal{M}$ for the full subcategory of $\mbfs \mathcal{M}$ spanned by the simplicial resolutions.

\item A \strong{cosimplicial resolution} in $\mathcal{M}$ is an weakly constant simplicial object in $\mathcal{M}$ that is also Reedy-cofibrant in $\mbfc \mathcal{M}$.

We write $\mbfc_\mathrm{r} \mathcal{M}$ for the full subcategory of $\mbfs \mathcal{M}$ spanned by the cosimplicial resolutions.
\end{itemize}
\end{dfn}

\begin{cor}
\ \noprelistbreak
\begin{enumerate}[(i)]
\item The adjunction
\[
\skel[0] \dashv \pblank_0 : \mbfs_\mathrm{w} \mathcal{M} \to \mathcal{M}
\]
is a adjoint homotopical equivalence of homotopical categories.

\item The induced adjunction
\[
\Ho \skel[0] \dashv \Ho \pblank_0 : \Ho \mbfs_\mathrm{w} \mathcal{M} \to \Ho \mathcal{M}
\]
is an adjoint equivalence of categories.
\end{enumerate}
Dually:
\begin{enumerate}[(i)]
\item The adjunction
\[
\pblank^0 \dashv \ccoskel[0] : \mathcal{M} \to \mbfc_\mathrm{w} \mathcal{M}
\]
is an adjoint homotopical equivalence of homotopical categories.

\item The induced adjunction
\[
\Ho \pblank^0 \dashv \Ho \ccoskel[0] : \Ho \mathcal{M} \to \Ho \mbfc_\mathrm{w} \mathcal{M}
\]
is an adjoint equivalence of categories.
\end{enumerate}
\end{cor}
\begin{proof}
This is an immediate consequence of the definitions and \autoref{prop:simplicial.object.Quillen.adjunction}.
\end{proof}

\begin{dfn}
\ \noprelistbreak
\begin{itemize}
\item Let $A$ be an object in $\mathcal{M}$ and let $B_{\bullet}$ be a simplicial object in $\mathcal{M}$. The \strong{right hom-complex} $\sHom[\mathcal{M}]{A}{B}$ is the simplicial set defined by the following formula:
\[
\sHom[\mathcal{M}]{A}{B}_n = \Hom[\mathcal{M}]{A}{B_n}
\]

\item Let $A^{\bullet}$ be a cosimplicial object in $\mathcal{M}$ and let $B$ be an object in $\mathcal{M}$. The \strong{left hom-complex} $\sHom[\mathcal{M}]{A}{B}$ is the simplicial set defined by the following formula:
\[
\sHom[\mathcal{M}]{A}{B}_n = \Hom[\mathcal{M}]{A^n}{B}
\]

\item Let $A^{\bullet}$ be a cosimplicial object in $\mathcal{M}$ and let $B_{\bullet}$ be a simplicial object in $\mathcal{M}$. The \strong{total hom-complex} $\sHom[\mathcal{M}]{A}{B}$ is the simplicial set defined by the following formula:
\[
\sHom[\mathcal{M}]{A}{B}_n = \Hom[\mathcal{M}]{A^n}{B_n}
\]
\end{itemize}
\end{dfn}

\begin{lem}
\label{lem:one-sided.homotopy.function.complex.homotopical.properties}
\ \noprelistbreak
\begin{itemize}
\item If $B_{\bullet}$ is a simplicial resolution in $\mathcal{M}$, then the right hom-complex functor $\sHom[\mathcal{M}]{\blank}{B} : \op{\mathcal{M}} \to \cat{\SSet}$ preserves weak equivalences between cofibrant objects.

\item If $A^{\bullet}$ is a cosimplicial resolution in $\mathcal{M}$, then the left hom-complex functor $\sHom[\mathcal{M}]{A}{\blank} : \mathcal{M} \to \cat{\SSet}$ preserves weak equivalences between fibrant objects.
\end{itemize}
\end{lem}
\begin{proof} \openproof
See Corollaries 6.3 and 6.4 in \citep{Dwyer-Kan:1980c}, Corollary 5.4.4 in \citep{Hovey:1999}, or Corollary 16.5.5 in \citep{Hirschhorn:2003}.
\end{proof}

\begin{cor}
\label{cor:two-sided.homotopy.function.complex.homotopical.properties}
The total hom-complex functor
\[
\sHom[\mathcal{M}]{\blank}{\blank} : \op{\parens{\mbfc_\mathrm{r} \mathcal{M}}} \times \mbfs_\mathrm{r} \mathcal{M} \to \cat{\SSet}
\]
preserves weak equivalences.
\end{cor}
\begin{proof}
Since Reedy-fibrant simplicial objects (\resp Reedy-cofibrant cosimplicial objects) are degreewise fibrant (\resp cofibrant), the claim is a consequence of lemmas~\ref{lem:Bousfield-Kan.homotopical.properties}, \ref{lem:diagonal.is.hocolim}, and~\ref{lem:one-sided.homotopy.function.complex.homotopical.properties}.
\end{proof}

\begin{thm}
\ \noprelistbreak
\begin{enumerate}[(i)]
\item The right hom-complex functor $\sHom[\mathcal{M}]{\blank}{\blank} : \op{\mathcal{M}} \times \mbfs_\mathrm{w} \mathcal{M} \to \cat{\SSet}$ is a right-deformable functor.

\item The left hom-complex functor $\sHom[\mathcal{M}]{\blank}{\blank} : \op{\parens{\mbfc_\mathrm{w} \mathcal{M}}} \times \mathcal{M} \to \cat{\SSet}$ is a right-deformable functor. 

\item The total hom-complex functor $\sHom[\mathcal{M}]{\blank}{\blank} : \op{\parens{\mbfc_\mathrm{w} \mathcal{M}}} \times \mbfs_\mathrm{w} \mathcal{M} \to \cat{\SSet}$ is a right-deformable functor. 
\end{enumerate}
In particular, each of the above-mentioned functors has a total right derived functor.
\end{thm}
\begin{proof}
The Reedy-fibrant replacement functor for $\mbfs \mathcal{M}$ (\resp Reedy-cofibrant replacement functor for $\mbfc \mathcal{M}$) restricts to a right (\resp left) deformation retract for $\mbfs_\mathrm{w} \mathcal{M}$ (\resp $\mbfc_\mathrm{w} \mathcal{M}$). The right-deformability of the functors in question then follows by \autoref{lem:one-sided.homotopy.function.complex.homotopical.properties} and \autoref{cor:two-sided.homotopy.function.complex.homotopical.properties}, and the existence of total right derived functors is an application of paragraph 41.5 in \citep{DHKS}.
\end{proof}

Recall that the \strong{totalisation} of a cosimplicial simplicial set $X^{\bullet}$ is the simplicial set $\Tot X^{\bullet}$ defined by the following end formula,
\[
\Tot X^{\bullet} = \int_{\bracket{m} : \cat{\Simplex}} \xHom{\Delta^m}{X^m}
\]
where $\Delta^m$ is the standard $m$-simplex and $\xHom{\blank}{\blank}$ denotes the internal hom of $\cat{\SSet}$.

\begin{prop}
\ \noprelistbreak
\begin{itemize}
\item The category $\mbfs \mathcal{M}$ admits a simplicial enrichment with hom-spaces defined by the following formula,
\[
\ulHom[\mbfs \mathcal{M}]{A}{B} = \Tot \sHom[\mathcal{M}]{A_{\bullet}}{B}
\]
where $A_{\bullet}$ and $B_{\bullet}$ are simplicial objects in $\mathcal{M}$ and $\sHom[\mathcal{M}]$ in the RHS denotes the right hom-complex.

\item The category $\mbfc \mathcal{M}$ admits a simplicial enrichment with hom-spaces defined by the following formula,
\[
\ulHom[\mbfc \mathcal{M}]{A}{B} = \Tot \sHom[\mathcal{M}]{A}{B^{\bullet}}
\]
where $A^{\bullet}$ and $B^{\bullet}$ are cosimplicial objects in $\mathcal{M}$ and $\sHom[\mathcal{M}]$ in the RHS denotes the left hom-complex.
\end{itemize}
\end{prop}
\begin{proof} \openproof
Omitted.
\end{proof}

Though this simplicial enrichment of $\mbfs \mathcal{M}$ (\resp $\mbfc \mathcal{M}$) usually fails to make it a simplicial model category, it has just enough good properties to ensure that the hom-space functor of $\mbfs_\mathrm{w} \mathcal{M}$ (\resp $\mbfc_\mathrm{w} \mathcal{M}$) admits a total right derived functor. Indeed:

\begin{thm}
\ \noprelistbreak
\begin{itemize}
\item The functor $\ulHom[\mbfs_\mathrm{w} \mathcal{M}]{\blank}{\blank} : \op{\parens{\mbfs_\mathrm{w} \mathcal{M}}} \times \mbfs_\mathrm{w} \mathcal{M} \to \cat{\SSet}$ is a right-deformable functor and has a total right derived functor.

\item The functor $\ulHom[\mbfc_\mathrm{w} \mathcal{M}]{\blank}{\blank} : \op{\parens{\mbfc_\mathrm{w} \mathcal{M}}} \times \mbfc_\mathrm{w} \mathcal{M} \to \cat{\SSet}$ is a right-deformable functor and has a total right derived functor.
\end{itemize}
\end{thm}
\begin{proof}
The two claims are formally dual; we will prove the first version.

Let $\mathcal{M}_\mathrm{c}$ be the full subcategory of $\mathcal{M}$ spanned by the cofibrant objects. Observe that for any object $A$ in $\mathcal{M}$ and any simplicial object $B_{\bullet}$ in $\mathcal{M}$, there is a natural isomorphism
\[
\ulHom[\mbfs_\mathrm{w} \mathcal{M}]{\skel[0]{A}}{B} \cong \sHom[\mathcal{M}]{A}{B}
\]
and so, by \autoref{cor:two-sided.homotopy.function.complex.homotopical.properties}, the functor $\ulHom[\mbfs_\mathrm{w} \mathcal{M}]{\skel[0]{\blank}}{\blank} : \op{\parens{\mathcal{M}_\mathrm{c}}} \times \mbfs_\mathrm{r} \mathcal{M} \to \cat{\SSet}$ preserves weak equivalences. But for every weakly constant simplicial object $A_{\bullet}$ in $\mathcal{M}$, there is a functorial choice of a cofibrant object $\tilde{A}$ in $\mathcal{M}$ and a Reedy weak equivalence $\skel[0]{\smash{\tilde{A}}} \to A_{\bullet}$, and for every weakly constant simplicial object $B_{\bullet}$, there is a functorial choice of a simplicial resolution $\hat{B}_{\bullet}$ and a Reedy weak equivalence $B_{\bullet} \to \hat{B}_{\bullet}$, so $\ulHom[\mbfs_\mathrm{w} \mathcal{M}]{\blank}{\blank} : \op{\parens{\mbfs_\mathrm{w} \mathcal{M}}} \times \mbfs_\mathrm{w} \mathcal{M} \to \cat{\SSet}$ is indeed right-deformable.
\end{proof}

In view of the results of this section, it seems reasonable to make the following definition:

\begin{dfn}
A \strong{derived hom-space functor} for $\mathcal{M}$ is a functor
\[
\RHom[\mathcal{M}] : \op{\Ho \mathcal{M}} \times \Ho \mathcal{M} \to \Ho \cat{\SSet}
\]
equipped with an isomorphism
\[
\RHom[\mathcal{M}]{\pblank^0}{\pblank_0} \cong \sHom[\mathcal{M}]{\blank}{\blank}
\]
of functors $\op{\parens{\Ho \mbfc_\mathrm{r} \mathcal{M}}} \times \Ho \mbfs_\mathrm{r} \mathcal{M} \to \Ho \cat{\SSet}$.
\end{dfn}

\section{Comparison with the hammock localisation}

Let $\mathcal{M}$ be a small model category and let $\mathcal{W}$ be the subcategory of weak equivalences. Recall the following definitions from \citep{Dwyer-Kan:1980b}:

\begin{dfn}
\ \noprelistbreak
\begin{itemize}
\item A \strong{hammock} in $\mathcal{M}$ from $A$ to $B$ of width $k$ and length $n$ is a commutative diagram in $\mathcal{M}$ of the form below,
\[
\hspace{-0.5in}
\begin{tikzcd}[column sep=4.5ex, row sep=4.5ex]
A \dar[equals] \rar[-] &
C_{0, 1} \dar \rar[-] &
C_{0, 2} \dar \rar[-] &
\cdots \rar[-] &
C_{0, n-2} \dar \rar[-] &
C_{0, n-1} \dar \rar[-] &
B \dar[equals] \\
A \dar[equals] \rar[-] &
C_{1, 1} \dar \rar[-] &
C_{1, 2} \dar \rar[-] &
\cdots \rar[-] &
C_{1, n-2} \dar \rar[-] &
C_{1, n-1} \dar \rar[-] &
B \dar[equals] \\
\vdots \dar[equals] &
\vdots \dar &
\vdots \dar &
\ddots &
\vdots \dar &
\vdots \dar &
\vdots \dar[equals] \\
A \dar[equals] \rar[-] &
C_{k-1, 1} \dar \rar[-] &
C_{k-1, 2} \dar \rar[-] &
\cdots \rar[-] &
C_{k-1, n-2} \dar \rar[-] &
C_{k-1, n-1} \dar \rar[-] &
B \dar[equals] \\
A \rar[-] &
C_{k, 1} \rar[-] &
C_{k, 2} \rar[-] &
\cdots \rar[-] &
C_{k, n-2} \rar[-] &
C_{k, n-1} \rar[-] &
B
\end{tikzcd}
\hspace{-0.5in}
\]
such that the following conditions are satisfied:
\begin{itemize}
\item In each column, all horizontal arrows point in the same direction.

\item All leftward-pointing arrows are weak equivalences.

\item All vertical arrows are weak equivalences.
\end{itemize}
We allow both $k$ and $n$ to be zero; if $n = 0$ then we must have $A = B$. 

\needspace{3\baselineskip}
\item A \strong{reduced hammock} in $\mathcal{M}$ is a hammock with these additional properties:
\begin{itemize}
\item In each column, not every horizontal arrow is an identity morphism.

\item Horizontal arrows in adjacent columns point in opposite directions.
\end{itemize}

\item The \strong{hammock localisation} of $\mathcal{M}$ is the following simplicially enriched category $\ul{\LH \mathcal{M}}$:
\begin{itemize}
\item The objects in $\LH \mathcal{M}$ are the objects in $\mathcal{M}$.

\item The hom-space $\ulHom[\LH \mathcal{M}]{A}{B}$ is the evident simplicial set whose $k$-simplices are the \emph{reduced} hammocks from $A$ to $B$ of width $k$ and \emph{any} length.

\item Composition is (horizontal) concatenation and identities are the hammocks of length $0$.
\end{itemize}
\end{itemize}
\end{dfn}

\needspace{3\baselineskip}
We are especially interested in the following:

\begin{dfn}
A \strong{special hammock} in $\mathcal{M}$ from $A$ to $B$ is a hammock of the form below,
\[
\begin{tikzcd}
A \dar[equals] \rar[leftarrow] &
\bullet \dar \rar &
\bullet \dar \rar[leftarrow] &
B \dar[equals] \\
A \dar[equals] \rar[leftarrow] &
\bullet \dar \rar &
\bullet \dar \rar[leftarrow] &
B \dar[equals] \\
A \dar[equals] &
\vdots \dar &
\vdots \dar &
B \dar[equals] \\
A \dar[equals] \rar[leftarrow] &
\bullet \dar \rar &
\bullet \dar \rar[leftarrow] &
B \dar[equals] \\
A \rar[leftarrow] &
\bullet \rar &
\bullet \rar[leftarrow] &
B
\end{tikzcd}
\]
where the horizontal arrows in the leftmost column are trivial fibrations and the horizontal arrows in the rightmost column are trivial cofibrations.

We write $\mathcal{T} \argp{A, B}$ for the following category:
\begin{itemize}
\item The objects are special hammocks in $\mathcal{M}$ from $A$ to $B$ of width $0$.

\item The morphisms are special hammocks in $\mathcal{M}$ from $A$ to $B$ of width $1$, with the top row as the domain and the bottom row as the codomain.

\item Composition and identities are inherited from $\mathcal{M}$.
\end{itemize}
\end{dfn}

\begin{remark}
Recalling that the class of trivial fibrations (\resp trivial cofibrations) in $\mathcal{M}$ is closed under pullback (\resp pushout), there is an evident pseudofunctor $\op{\mathcal{M}} \times \mathcal{M} \to \bicat{\Cat}$ whose value at $\tuple{A, B}$ is the category $\mathcal{T} \argp{A, B}$.
\end{remark}

\begin{lem}
\label{lem:natural.morphism.of.special.hammocks}
The obvious morphism $\nv{\mathcal{T} \argp{A, B}} \to \ulHom[\LH \mathcal{M}]{A}{B}$ is natural in the following sense: given morphisms $A' \to A$ and $B \to B'$ in $\mathcal{M}$, the following diagram commutes in $\Ho \cat{\SSet}$,
\[
\begin{tikzcd}
\nv{\mathcal{T} \argp{A, B}} \dar \rar &
\ulHom[\LH \mathcal{M}]{A}{B} \dar \\
\nv{\mathcal{T} \argp{A', B'}} \rar &
\ulHom[\LH \mathcal{M}]{A'}{B'}
\end{tikzcd}
\]
where the vertical arrows are the evident induced morphisms.
\end{lem}
\begin{proof}
By pasting commutative diagrams, we may reduce to the case where either $A' \to A$ or $B \to B'$ is an identity morphism, which is straightforward.
\end{proof}

\begin{prop}
\label{prop:special.hammocks.vs.reduced.hammocks}
The obvious morphism $\nv{\mathcal{T} \argp{A, B}} \to \ulHom[\LH \mathcal{M}]{A}{B}$ is a weak homotopy equivalence.
\end{prop}
\begin{proof} \openproof
It is straightforward (using the functorial factorisations of $\mathcal{M}$) to show that  the inclusion $\mathcal{T} \argp{A, B} \embedinto \inv{\mathcal{W}} \mathcal{M} \inv{\mathcal{W}} \argp{A, B}$ induces a weak homotopy equivalence of nerves. The claim is then a consequence of Propositions 6.2 and 8.2 in \citep{Dwyer-Kan:1980b}.
\end{proof}

\begin{prop}
\label{prop:cofinality.of.resolutions}
\ \noprelistbreak
\begin{itemize}
\item Let $B$ be an object in $\mathcal{M}$, let $\hat{B}_{\bullet}$ be a simplicial resolution in $\mathcal{M}$, let $i_{\bullet} : \skel[0]{B} \to \hat{B}_{\bullet}$ be a degreewise trivial cofibration, and let $\parens{\undercat{B}{\mathcal{W}}}_\mathrm{c}$ be the full subcategory of the slice category $\undercat{B}{\mathcal{W}}$ spanned by the trivial cofibrations with domain $B$. Then the diagram $I : \op{\cat{\Simplex}} \to \parens{\undercat{B}{\mathcal{W}}}_\mathrm{c}$ corresponding to $i_{\bullet}$ is a left aspherical functor.

\item Let $A$ be an object in $\mathcal{M}$, let $\tilde{A}^{\bullet}$ be a cosimplicial resolution in $\mathcal{M}$, let $p^{\bullet} : \tilde{A}^{\bullet} \to \ccoskel[0]{A}$ be a degreewise trivial fibration, and let $\parens{\overcat{\mathcal{W}}{A}}_\mathrm{f}$ be the full subcategory of the slice category $\overcat{\mathcal{W}}{A}$ spanned by the trivial fibrations with codomain $A$. Then the diagram $P : \cat{\Simplex} \to \parens{\overcat{\mathcal{W}}{A}}_\mathrm{f}$ corresponding to $p^{\bullet}$ is a right aspherical functor.
\end{itemize}
\end{prop}
\begin{proof} \openproof
This is essentially Propositions 6.11 and 6.12 in \citep{Dwyer-Kan:1980c}.
\end{proof}

\begin{prop}
\label{prop:homotopy.function.complexes.vs.special.hammocks}
\ \noprelistbreak
\begin{enumerate}[(i)]
\item Let $A$ and $B$ be objects in $\mathcal{M}$, let $\tilde{A}^{\bullet}$ be a cosimplicial resolution in $\mathcal{M}$, let $\hat{B}_{\bullet}$ be a simplicial resolution in $\mathcal{M}$, let $p^{\bullet} : \tilde{A}^{\bullet} \to \ccoskel[0]{A}$ be a degreewise trivial fibration, and let $i_{\bullet} : \skel[0]{B} \to \hat{B}_{\bullet}$ be a degreewise trivial cofibration. Then we have a diagram of weak homotopy equivalences of the form below:
\[
\begin{tikzcd}[column sep=0.0ex, row sep=3.0ex]
{} &
\KBcolim[n : \op{\cat{\Simplex}}] \BKcolim[m : \op{\cat{\Simplex}}] \discr \Hom[\mathcal{M}]{\tilde{A}^n}{\hat{B}_m} \dlar \drar \\
\sHom[\mathcal{M}]{\tilde{A}}{\hat{B}} &&
\nv{\mathcal{T} \argp{A, B}}
\end{tikzcd}
\]

\item Moreover, the above diagram is natural in the following sense: given commutative diagrams in $\mbfc \mathcal{M}$ and $\mbfs \mathcal{M}$ of the forms below,
\[
\mmhfill
\begin{tikzcd}
\ccoskel[0]{A} \dar[leftarrow] \rar[leftarrow]{p^{\bullet}} &
\tilde{A}^{\bullet} \dar[leftarrow] \\
\ccoskel[0]{A'} \rar[leftarrow, swap]{p^{\prime \bullet}} &
\tilde{A}^{\prime \bullet}
\end{tikzcd}
\mmhfill
\begin{tikzcd}
\hat{B}_{\bullet} \dar \rar[leftarrow]{i_{\bullet}} &
\skel[0]{B} \dar \\
\hat{B}' \rar[leftarrow, swap]{i'_{\bullet}} &
\skel[0]{B'}
\end{tikzcd}
\mmhfill
\]
the following diagram commutes in $\Ho \cat{\SSet}$,
\[
\begin{tikzcd}[column sep=0.0ex, row sep=3.0ex]
{} &
\KBcolim[n : \op{\cat{\Simplex}}] \BKcolim[m : \op{\cat{\Simplex}}] \discr \Hom[\mathcal{M}]{\tilde{A}^n}{\hat{B}_m} \arrow{dd} \dlar \drar \\
\sHom[\mathcal{M}]{\tilde{A}}{\hat{B}} \arrow{dd} &&
\nv{\mathcal{T} \argp{A, B}} \arrow{dd} \\
&
\KBcolim[n : \op{\cat{\Simplex}}] \BKcolim[m : \op{\cat{\Simplex}}] \discr \Hom[\mathcal{M}]{\tilde{A}^{\prime n}}{\hat{B}'_m} \dlar \drar \\
\sHom[\mathcal{M}]{\tilde{A}'}{\hat{B}'} &&
\nv{\mathcal{T} \argp{A', B'}}
\end{tikzcd}
\]
where the vertical arrows are the evident induced morphisms.
\end{enumerate}
\end{prop}
\begin{proof}
(i). We follow paragraph 7.2 in \citep{Dwyer-Kan:1980c}. By applying \autoref{lem:diagonal.is.hocolim} (twice), we obtain a natural weak homotopy equivalence of the following type:
\[
\textstyle \KBcolim[\op{\cat{\Simplex}}] \BKcolim[\op{\cat{\Simplex}}] \discr \Hom[\mathcal{M}]{\tilde{A}^{\bullet}}{\hat{B}_{\bullet}} \to \sHom[\mathcal{M}]{\tilde{A}}{\hat{B}}
\]
On the other hand, by \autoref{lem:Bousfield-Kan.homotopical.properties} and Thomason's homotopy colimit theorem (\ref{thm:Thomason.hocolim}), we have a natural weak homotopy equivalence
\[
\textstyle \KBcolim[\op{\cat{\Simplex}}] \BKcolim[\op{\cat{\Simplex}}] \discr \Hom[\mathcal{M}]{\tilde{A}^{\bullet}}{\hat{B}_{\bullet}} \to \nv{\oplaxcolim[\op{\cat{\Simplex}}] \laxcolim[\op{\cat{\Simplex}}] \discr \Hom[\mathcal{M}]{\tilde{A}^{\bullet}}{\hat{B}_{\bullet}}}
\]
and recalling Quillen's Theorem A (\ref{thm:Quillen-A}) and the homotopy cofinality theorem (\ref{thm:Grothendieck.homotopy.cofinality}), \autoref{prop:cofinality.of.resolutions} implies there is a weak homotopy equivalence
\[
\hspace{-0.5in}
\textstyle \nv{\oplaxcolim[\op{\cat{\Simplex}}] \laxcolim[\op{\cat{\Simplex}}] \discr \Hom[\mathcal{M}]{\tilde{A}^{\bullet}}{\hat{B}_{\bullet}}} \to \nv{\oplaxcolim[\op{\parens{\overcat{\mathcal{W}}{A}}_\mathrm{f}}] \laxcolim[\parens{\undercat{B}{\mathcal{W}}}_\mathrm{c}] \discr \Hom[\mathcal{M}]{Q}{R}}
\hspace{-0.5in}
\]
where $Q : \parens{\overcat{\mathcal{W}}{A}}_\mathrm{f} \to \mathcal{M}$ and $R : \parens{\undercat{B}{\mathcal{W}}}_\mathrm{c} \to \mathcal{M}$ are the evident projection functors; but it is straightforward to check that
\[
\textstyle \oplaxcolim[\op{\parens{\overcat{\mathcal{W}}{A}}_\mathrm{f}}] \laxcolim[\parens{\undercat{B}{\mathcal{W}}}_\mathrm{c}] \discr \Hom[\mathcal{M}]{Q}{R} \cong \mathcal{T} \argp{A, B}
\]
so we are done. 

\bigskip\noindent
(ii). Naturality implies that the left half of the diagram in question commutes strictly, \ie
\[
\begin{tikzcd}
\KBcolim[\op{\cat{\Simplex}}] \BKcolim[\op{\cat{\Simplex}}] \discr \Hom[\mathcal{M}]{\tilde{A}^{\bullet}}{\hat{B}_{\bullet}} \dar \rar &
\sHom[\mathcal{M}]{\tilde{A}}{\hat{B}} \dar \\
\KBcolim[\op{\cat{\Simplex}}] \BKcolim[\op{\cat{\Simplex}}] \discr \Hom[\mathcal{M}]{\tilde{A}^{\prime \bullet}}{\hat{B}'_{\bullet}} \rar &
\sHom[\mathcal{M}]{\tilde{A}'}{\hat{B}'}
\end{tikzcd}
\]
commutes in $\cat{\SSet}$; and similarly,
\[
\begin{tikzcd}
\KBcolim[\op{\cat{\Simplex}}] \BKcolim[\op{\cat{\Simplex}}] \discr \Hom[\mathcal{M}]{\tilde{A}^{\bullet}}{\hat{B}_{\bullet}} \dar \rar &
\nv{\oplaxcolim[\op{\cat{\Simplex}}] \laxcolim[\op{\cat{\Simplex}}] \discr \Hom[\mathcal{M}]{\tilde{A}^{\bullet}}{\hat{B}_{\bullet}}} \dar \\
\KBcolim[\op{\cat{\Simplex}}] \BKcolim[\op{\cat{\Simplex}}] \discr \Hom[\mathcal{M}]{\tilde{A}^{\prime \bullet}}{\hat{B}'_{\bullet}} \rar &
\nv{\oplaxcolim[\op{\cat{\Simplex}}] \laxcolim[\op{\cat{\Simplex}}] \discr \Hom[\mathcal{M}]{\tilde{A}^{\prime \bullet}}{\hat{B}'_{\bullet}}}
\end{tikzcd}
\]
also commutes in $\cat{\SSet}$, so it suffices to verify that the evident diagram
\[
\begin{tikzcd}
\nv{\oplaxcolim[\op{\cat{\Simplex}}] \laxcolim[\op{\cat{\Simplex}}] \discr \Hom[\mathcal{M}]{\tilde{A}^{\prime \bullet}}{\hat{B}'_{\bullet}}} \dar \rar &
\nv{\mathcal{T} \argp{A, B}} \dar \\
\nv{\oplaxcolim[\op{\cat{\Simplex}}] \laxcolim[\op{\cat{\Simplex}}] \discr \Hom[\mathcal{M}]{\tilde{A}^{\prime \bullet}}{\hat{B}'_{\bullet}}} \rar &
\nv{\mathcal{T} \argp{A', B'}}
\end{tikzcd}
\]
commutes in $\Ho \cat{\SSet}$. By pasting commutative diagrams, we may reduce the problem to the following two cases:
\begin{itemize}
\item Both $\tilde{A}^{\prime \bullet} \to \tilde{A}^{\bullet}$ and $A' \to A$ are identity morphisms.

\item Both $\hat{B}_{\bullet} \to \hat{B}'_{\bullet}$ and $B \to B'$ are identity morphisms.
\end{itemize}
Furthermore, the two cases are formally dual, so it is enough to check the first case. But the universal property of pushouts yields a natural transformation fitting into the diagram below,
\[
\begin{tikzpicture}[commutative diagrams/every diagram]
\matrix[matrix of math nodes, name=m]{
	\oplaxcolim[\op{\cat{\Simplex}}] \laxcolim[\op{\cat{\Simplex}}] \discr \Hom[\mathcal{M}]{\tilde{A}^{\bullet}}{\hat{B}_{\bullet}} &
	\mathcal{T} \argp{A, B} \\
	\oplaxcolim[\op{\cat{\Simplex}}] \laxcolim[\op{\cat{\Simplex}}] \discr \Hom[\mathcal{M}]{\tilde{A}^{\bullet}}{\hat{B}'_{\bullet}} &
	\mathcal{T} \argp{A, B'} \\
};
\path[commutative diagrams/.cd, every arrow, every label]
	(m-1-1) edge (m-2-1)
	(m-1-1) edge (m-1-2)
	(m-1-2) edge (m-2-2)
	(m-2-1) edge (m-2-2);
\path (m-1-2) -- (m-2-1)
	node[near start](a-1){} node[near end](b-1){};
\path[commutative diagrams/.cd, every arrow, every label]
	(a-1) edge[commutative diagrams/Rightarrow] (b-1);
\end{tikzpicture}
\]
so we are done.
\end{proof}

\begin{thm}
\label{thm:derived.hom.spaces.via.hammock.localisation}
There is an isomorphism
\[
\RHom[\mathcal{M}]{\blank}{\blank} \cong \ulHom[\LH \mathcal{M}]{\blank}{\blank}
\]
of functors $\op{\Ho \mathcal{M}} \times \Ho \mathcal{M} \to \Ho \cat{\SSet}$.
\end{thm}
\begin{proof}
Combine \autoref{lem:natural.morphism.of.special.hammocks} and propositions~\ref{prop:special.hammocks.vs.reduced.hammocks} and~\ref{prop:homotopy.function.complexes.vs.special.hammocks}.
\end{proof}

\section{Bousfield localisation and simplicial localisation}

Let $\mathcal{C}$ be a small category and let $\mathcal{W}$ be a subcategory of weak equivalences. Recall the following definitions from \citep{Dwyer-Kan:1980a}:

\begin{dfn}
\ \noprelistbreak
\begin{itemize}
\item The \strong{standard resolution} of a category $\mathcal{A}$ is the simplicial category $F_{\bullet} \mathcal{A}$, where the $0$-th level is the free category generated by the underlying reflexive graph of $\mathcal{A}$ and the $\parens{n + 1}$-th level is the free category generated by the underlying reflexive graph of the $n$-th level.

\item The \strong{simplicial localisation} of $\mathcal{C}$ is the simplicially enriched category $\ul{\totalL \mathcal{C}}$ corresponding to the simplicial category $F_{\bullet} \mathcal{C} \argb{\inv{F_{\bullet} \mathcal{W}}}$ obtained by inverting $F_{\bullet} \mathcal{W}$ in $F_{\bullet} \mathcal{C}$ levelwise.
\end{itemize}
\end{dfn}

\begin{prop}
\label{prop:local.model.structure.on.simplicial.presheaves}
Let $\homs : \mathcal{C} \to \Func{\op{\mathcal{C}}}{\cat{\SSet}}$ be the Yoneda embedding, \ie the functor defined by $\homs{B} = \discr \Hom[\mathcal{C}]{\blank}{B}$.
\begin{enumerate}[(i)]
\item The projective model structure on $\Func{\op{\mathcal{C}}}{\cat{\SSet}}$ exists.

\item The left Bousfield localisation of the projective model structure with respect to  $\set{ \homs{w} }{ w \in \mor \mathcal{W} }$ exists.

\item A simplicial presheaf $P : \op{\mathcal{C}} \to \cat{\SSet}$ is fibrant in the localised model structure if and only if $P$ is projective-fibrant and sends weak equivalences in $\mathcal{C}$ to weak homotopy equivalences.
\end{enumerate}
\end{prop}
\begin{proof}
(i). Apply Theorem 11.6.1 in \citep{Hirschhorn:2003}.

\bigskip\noindent
(ii). Apply Theorem 4.1.1 in \citep{Hirschhorn:2003}.

\bigskip\noindent
(iii). By Proposition 3.4.1 in \citep{Hirschhorn:2003}, $P$ is fibrant in the localised model structure if and only if it is a local object; and by Example 17.2.4 in \opcit (plus the enriched Yoneda lemma), $P$ is a local object if and only if it is projective-fibrant and sends weak equivalences in $\mathcal{C}$ to weak homotopy equivalences.
\end{proof}

The localised model structure on $\Func{\op{\mathcal{C}}}{\cat{\SSet}}$ allows us to find the best approximation of an arbitrary simplicial presheaf $\op{\mathcal{C}} \to \cat{\SSet}$ by one that sends weak equivalences in $\mathcal{C}$ to weak homotopy equivalences. More precisely:
%

\begin{prop}
Let $R : \Func{\op{\mathcal{C}}}{\cat{\SSet}} \to \Func{\op{\mathcal{C}}}{\cat{\SSet}}$ be a fibrant replacement functor for the localised model structure and let $i : \id \hoto R$ be a natural weak equivalence in the localised model structure. Then for any simplicial presheaf $X : \op{\mathcal{C}} \to \cat{\SSet}$, $\tuple{R X, i_X}$ is a right approximation for $X$.
\end{prop}
\begin{proof}
Let $Y : \op{\mathcal{C}} \to \cat{\SSet}$ be a simplicial presheaf and suppose $Y$ sends weak equivalences in $\mathcal{C}$ to weak homotopy equivalences. First, let us show that $i_Y : Y \to R Y$ is a weak equivalence in the projective model structure on $\Func{\op{\mathcal{C}}}{\cat{\SSet}}$. Let $j : Y \to \hat{Y}$ be any weak equivalence in the projective model structure where $\hat{Y}$ is projective-fibrant. Then the following diagram in $\Func{\op{\mathcal{C}}}{\cat{\SSet}}$ commutes,
\[
\begin{tikzcd}
Y \dar[swap]{j} \rar{i_Y} &
R Y \dar{R j} \\
\hat{Y} \rar[swap]{i_{\hat{Y}}} &
R \hat{Y}
\end{tikzcd}
\]
and by \autoref{prop:local.model.structure.on.simplicial.presheaves}, $\hat{Y}$ is a fibrant object in the localised model structure on $\Func{\op{\mathcal{C}}}{\cat{\SSet}}$, so both $i_{\hat{Y}} : \hat{Y} \to R \hat{Y}$ and $R j : R Y \to R \hat{Y}$ are weak equivalences between fibrant objects in the localised model structure, hence also weak equivalences between fibrant objects in the projective model structure by Theorem 3.2.13 in \citep{Hirschhorn:2003}. But $j : Y \to \hat{Y}$ is a weak equivalence in the projective model structure, so it follows that the same is true for $i_Y : Y \to R Y$.

Now, consider a morphism $\alpha : X \to Y$ in $\Func{\op{\mathcal{C}}}{\cat{\SSet}}$. Then the following diagram in $\Func{\op{\mathcal{C}}}{\cat{\SSet}}$ commutes,
\[
\begin{tikzcd}
X \dar[swap]{i_X} \rar[equals] &
X \dar \rar[equals] &
X \dar{\alpha} \\
R X \rar[swap]{R \alpha} &
R Y \rar[swap, leftarrow]{i_Y} &
Y 
\end{tikzcd}
\]
and if $\alpha : X \to Y$ is a weak equivalence in the localised model structure, then $R \alpha : R X \to R Y$ is a weak equivalence in the projective model structure. The diagram is clearly natural in $\alpha : X \to Y$, so $\tuple{R X, i_X}$ is a homotopically initial Kan extension of $X : \op{\mathcal{C}} \to \cat{\SSet}$ along $\id : \op{\mathcal{C}} \to \op{\mathcal{C}}$, \ie a right approximation for $X$.
\end{proof}

\begin{remark}
Unfortunately, it does not follow that every simplicial presheaf $\op{\mathcal{C}} \to \cat{\SSet}$ admits a total right derived functor; right approximations only have a universal property with respect to functors $\op{\Ho \mathcal{C}} \to \Ho \cat{\SSet}$ that arise from simplicial presheaves $\op{\mathcal{C}} \to \cat{\SSet}$.
\end{remark}

\begin{prop}
\label{prop:local.model.structure.on.simplicial.presheaves.on.the.standard.resolution}
Let $\ul{F \mathcal{C}}$ be the simplicially enriched category corresponding to the standard resolution $F_{\bullet} \mathcal{C}$ and let $\ulhoms : \ul{F \mathcal{C}} \to \Func{\op{\ul{F \mathcal{C}}}}{\ul{\cat{\SSet}}}$ be the enriched Yoneda embedding, \ie the simplicially enriched functor defined by $\ulhoms{B} = \ulHom[F \mathcal{C}]{\blank}{B}$.
\begin{enumerate}[(i)]
\item The projective model structure on $\Func{\op{\ul{F \mathcal{C}}}}{\ul{\cat{\SSet}}}$ exists.

\item The left Bousfield localisation of the projective model structure with respect to $\set{ \ulhoms{w} }{ w \in \mor F \mathcal{W} }$ exists.

\item A simplicial presheaf $\ul{P} : \op{\ul{F \mathcal{C}}} \to \ul{\cat{\SSet}}$ is fibrant in the localised model structure if and only if $P$ is projective-fibrant and sends morphisms in $F \mathcal{W}$ to weak homotopy equivalences.
\end{enumerate}
\end{prop}
\begin{proof}
(i). Apply Theorem 11.3.2 in \citep{Hirschhorn:2003} to the evident forgetful functor $\Func{\op{\ul{F \mathcal{C}}}}{\ul{\cat{\SSet}}} \to \Func{\ob \mathcal{C}}{\ul{\cat{\SSet}}}$.

\bigskip\noindent
(ii) and (iii). These may be proved the same way as in \autoref{prop:local.model.structure.on.simplicial.presheaves}.
\end{proof}

\begin{thm}
\label{thm:Quillen.equivalence.of.local.model.structures}
Let $\ul{U} : \ul{F \mathcal{C}} \to \mathcal{C}$ be the standard augmentation and let $\ul{V} : \ul{F \mathcal{C}} \to \ul{\totalL \mathcal{C}}$ be the localising functor.
\begin{enumerate}[(i)]
\item The induced functor
\[
\ul{U}^* : \Func{\op{\mathcal{C}}}{\ul{\cat{\SSet}}} \to \Func{\op{\ul{F \mathcal{C}}}}{\ul{\cat{\SSet}}}
\]
is a right Quillen equivalence with respect to the projective model structures.

\item The induced functor
\[
\ul{U}^* : \Func{\op{\mathcal{C}}}{\ul{\cat{\SSet}}} \to \Func{\op{\ul{F \mathcal{C}}}}{\ul{\cat{\SSet}}}
\]
is a right Quillen equivalence with respect to the localised model structures.

\item The induced functor
\[
\ul{V}^* : \Func{\op{\ul{\totalL \mathcal{C}}}}{\cat{\SSet}} \to \Func{\op{\ul{F \mathcal{C}}}}{\ul{\cat{\SSet}}}
\]
is a right Quillen equivalence with respect to the projective model structure on $\Func{\op{\ul{\totalL \mathcal{C}}}}{\ul{\cat{\SSet}}}$ and the localised model structure on $\Func{\op{\ul{F \mathcal{C}}}}{\ul{\cat{\SSet}}}$.
\end{enumerate}
\end{thm}
\begin{proof}
(i). First, we must show that $\ul{U}^* : \Func{\op{\mathcal{C}}}{\ul{\cat{\SSet}}} \to \Func{\op{\ul{F \mathcal{C}}}}{\ul{\cat{\SSet}}}$ is a right Quillen functor with respect to the projective model structures. It is well known that $\ul{U}^*$ has a left adjoint, namely the unique (up to unique isomorphism) simplicially enriched functor $\ul{U}_! : \Func{\op{\ul{F \mathcal{C}}}}{\ul{\cat{\SSet}}} \to \Func{\op{\mathcal{C}}}{\ul{\cat{\SSet}}}$ that preserves simplicially enriched colimits and makes the following diagram commute:
\[
\begin{tikzcd}
\ul{F \mathcal{C}} \dar[swap]{\ul{U}} \rar{\ulhoms} &
\Func{\op{\ul{F \mathcal{C}}}}{\ul{\cat{\SSet}}} \dar[dashed]{\ul{U}_!} \\
\mathcal{C} \rar[swap]{\homs} &
\Func{\op{\mathcal{C}}}{\ul{\cat{\SSet}}}
\end{tikzcd}
\]
Moreover, it is clear that $\ul{U}^*$ preserves projective fibrations and natural weak equivalences, so $\ul{U}^*$ is indeed a right Quillen functor.

It remains to be verified that the functor $\ul{U}^*$ is a right Quillen equivalence, and by Proposition 1.3.13 in \citep{Hovey:1999} it suffices to check that the right derived functor (with respect to the projective model structures)
\[
\totalR \ul{U}^* : \Ho \Func{\op{\mathcal{C}}}{\ul{\cat{\SSet}}} \to \Ho \Func{\op{\ul{F \mathcal{C}}}}{\ul{\cat{\SSet}}}
\]
is fully faithful and essentially surjective on objects. But $\ul{U} : \ul{F \mathcal{C}} \to \mathcal{C}$ is a Dwyer--Kan equivalence (by Proposition 2.6 in \citep{Dwyer-Kan:1980a}), so this is a straightforward consequence of Theorem 2.1 in \citep{Dwyer-Kan:1987a}.

\bigskip\noindent
(ii). We already know that $\ul{U}^* : \Func{\op{\mathcal{C}}}{\ul{\cat{\SSet}}} \to \Func{\op{\ul{F \mathcal{C}}}}{\ul{\cat{\SSet}}}$ is a right Quillen functor with respect to the projective model structures, so $\ul{U}_!$ is a left Quillen functor with respect to the projective model structures. Since representable simplicial presheaves are projective-cofibrant and $\ul{U}$ restricts to a functor $\ul{F \mathcal{W}} \to \mathcal{W}$, we may apply Proposition 3.3.18 in \citep{Hirschhorn:2003} and deduce that $\ul{U}_!$ is a left Quillen functor with respect to the localised model structures. Thus, $\ul{U}^*$ is indeed a right Quillen functor with respect to the localised model structures.

To show that $\ul{U}^*$ is a right Quillen equivalence, it now suffices to check that the right derived functor (with respect to the localised model structures)
\[
\totalR \ul{U}^* : \Ho \Func{\op{\mathcal{C}}}{\ul{\cat{\SSet}}} \to \Ho \Func{\op{\ul{F \mathcal{C}}}}{\ul{\cat{\SSet}}}
\]
is fully faithful and essentially surjective on objects. Recalling propositions~\ref{prop:local.model.structure.on.simplicial.presheaves} and~\ref{prop:local.model.structure.on.simplicial.presheaves.on.the.standard.resolution}, this is a straightforward consequence of Corollary 3.8 in \citep{Dwyer-Kan:1987a}.

\bigskip\noindent
(iii). As with (i), it is easy to see that $\ul{V}^* : \Func{\op{\ul{\totalL \mathcal{C}}}}{\cat{\SSet}} \to \Func{\op{\ul{F \mathcal{C}}}}{\ul{\cat{\SSet}}}$ is a right Quillen functor with respect to the projective model structures, so it is a right Quillen functor with respect to the localised model structure on $\Func{\op{\ul{F \mathcal{C}}}}{\ul{\cat{\SSet}}}$ \emph{a fortiori}. Thus, to show that $\ul{V}^*$ is a right Quillen equivalence, it suffices to check that the right derived functor
\[
\totalR \ul{V}^* : \Ho \Func{\op{\ul{\totalL \mathcal{C}}}}{\cat{\SSet}} \to \Ho \Func{\op{\ul{F \mathcal{C}}}}{\ul{\cat{\SSet}}}
\]
is fully faithful and essentially surjective on objects, and this is a consequence of paragraph 4.2 in \citep{Dwyer-Kan:1987a}.
\end{proof}

\begin{lem}
\label{lem:unit.of.left.Kan.extension}
\ \noprelistbreak
\begin{enumerate}[(i)]
\item Let $\ul{U}_! : \Func{\op{\ul{F \mathcal{C}}}}{\ul{\cat{\SSet}}} \to \Func{\op{\mathcal{C}}}{\cat{\SSet}}$ be the left adjoint of the functor $\ul{U}^* : \Func{\op{\mathcal{C}}}{\cat{\SSet}} \to \Func{\op{\ul{F \mathcal{C}}}}{\ul{\cat{\SSet}}}$. For any projective-cofibrant simplicial presheaf $\ul{X} : \op{\ul{F \mathcal{C}}} \to \ul{\cat{\SSet}}$, the adjunction unit
\[
\eta_{\ul{X}} : \ul{X} \to \ul{U}^* \ul{U}_! \ul{X}
\]
is a weak equivalence in the projective model structure on $\Func{\op{\ul{F \mathcal{C}}}}{\ul{\cat{\SSet}}}$.

\item Let $\ul{V}_! : \Func{\op{\ul{F \mathcal{C}}}}{\ul{\cat{\SSet}}} \to \Func{\op{\ul{\totalL \mathcal{C}}}}{\cat{\SSet}}$ be the left adjoint of the functor $\ul{V}^* : \Func{\op{\ul{\totalL \mathcal{C}}}}{\cat{\SSet}} \to \Func{\op{\ul{F \mathcal{C}}}}{\ul{\cat{\SSet}}}$. For any projective-cofibrant simplicial presheaf $\ul{X} : \op{\ul{F \mathcal{C}}} \to \ul{\cat{\SSet}}$, the adjunction unit
\[
\eta_{\ul{X}} : \ul{X} \to \ul{V}^* \ul{V}_! \ul{X}
\]
is a weak equivalence in the localised model structure on $\Func{\op{\ul{F \mathcal{C}}}}{\ul{\cat{\SSet}}}$.
\end{enumerate}
\end{lem}
\begin{proof}
(i). Let $i : \ul{U}_! \ul{X} \to \ul{Y}$ be any weak equivalence in projective model structure on $\Func{\op{\mathcal{C}}}{\cat{\SSet}}$ where $\ul{Y}$ is projective-fibrant. The functor $\ul{U}^* : \Func{\op{\mathcal{C}}}{\cat{\SSet}} \to \Func{\op{\ul{F \mathcal{C}}}}{\ul{\cat{\SSet}}}$ preserves all weak equivalences, so the morphism $\ul{U}^* i : \ul{U}^* \ul{U}_! \ul{X} \to \ul{U}^* \ul{Y}$ is a weak equivalence in the projective model structure on $\Func{\op{\ul{F \mathcal{C}}}}{\ul{\cat{\SSet}}}$; but \autoref{thm:Quillen.equivalence.of.local.model.structures} implies that the composite $\ul{U}^* i \circ \eta_{\ul{X}}$ is a weak equivalence in $\Func{\op{\ul{F \mathcal{C}}}}{\ul{\cat{\SSet}}}$, so the claim is a consequence of the 2-out-of-3 property.

\bigskip\noindent
(ii). A similar argument works.
\end{proof}
%
%

\begin{prop}
\label{prop:right.approximations.for.representable.presheaves}
Let $R : \Func{\op{\mathcal{C}}}{\cat{\SSet}} \to \Func{\op{\mathcal{C}}}{\cat{\SSet}}$ be a fibrant replacement functor for the localised model structure.
\begin{enumerate}[(i)]
\item $R \circ \homs : \mathcal{C} \to \Func{\op{\mathcal{C}}}{\cat{\SSet}}$ preserves weak equivalences.

\item There is an isomorphism
\[
R \circ \homs \cong \ulHom[\totalL \mathcal{C}]{\blank}{\blank}
\]
of functors $\op{\Ho \mathcal{C}} \times \Ho \mathcal{C} \to \Ho \cat{\SSet}$.
\end{enumerate}
\end{prop}
\begin{proof}
Let $S : \Func{\op{\ul{F \mathcal{C}}}}{\ul{\cat{\SSet}}} \to \Func{\op{\ul{F \mathcal{C}}}}{\ul{\cat{\SSet}}}$ be a fibrant replacement functor for the localised model structure. Then, for every object $B$ in $\mathcal{C}$, recalling that $\ul{U}_! \ulhoms{B} = \homs{B}$, we have the following commutative diagram in $\Func{\op{\ul{F \mathcal{C}}}}{\ul{\cat{\SSet}}}$,
\[
\begin{tikzcd}
\ul{V}^* \ul{V}_! \ulhoms{B} \dar & 
\ulhoms{B} \lar[swap]{\eta_{\ulhoms{B}}} \dar \rar{\eta_{\ulhoms{B}}} &
\ul{U}^* \ul{U}_! \ulhoms{B} \dar \rar &
\ul{U}^* R \homs{B} \dar \\
S \ul{V}^* \ul{V}_! \ulhoms{B} & 
S \ulhoms{B} \lar{R \eta_{\ulhoms{B}}} \rar[swap]{S \eta_{\ulhoms{B}}} &
S \ul{U}^* \ul{U}_! \ulhoms{B} \rar &
S \ul{U}^* R \homs{B}
\end{tikzcd}
\]
where the vertical arrows are natural weak equivalences in the localised model structure. \Autoref{thm:Quillen.equivalence.of.local.model.structures} and \autoref{lem:unit.of.left.Kan.extension} imply that the horizontal arrows are also weak equivalences in the localised model structure, and since weak equivalences between fibrant objects in the localised model structures are also weak equivalences between fibrant objects in the projective model structure, it follows that
\[
\ul{U}^* R \homs_B \cong S \ul{V}^* \ul{V}_! \ulhoms{B}
\]
as functors $\op{F \mathcal{C}} \to \Ho \cat{\SSet}$, naturally in $B$. Moreover, it is straightforward (using \autoref{prop:local.model.structure.on.simplicial.presheaves.on.the.standard.resolution}) to verify that the morphism $\ul{V}^* \ul{V}_! \ulhoms{B} \to S \ul{V}^* \ul{V}_! \ulhoms{B}$ is a weak equivalence in the projective model structure, so recalling that $\ul{V}_! \ulhoms{B} = \ulHom[\totalL \mathcal{C}]{\blank}{B}$, we have an isomorphism
\[
R \circ \homs \cong \ulHom[\totalL \mathcal{C}]{\blank}{\blank}
\]
of functors $\op{F \mathcal{C}} \times F \mathcal{C} \to \Ho \cat{\SSet}$. But these functors both factor through the evident functor $\op{F \mathcal{C}} \times F \mathcal{C} \to \op{\Ho \mathcal{C}} \times \Ho \mathcal{C}$, so we are done.
\end{proof}

\begin{thm}
\ \noprelistbreak
\begin{enumerate}[(i)]
\item There exist a functor $\hathoms : \mathcal{C} \to \Func{\op{\mathcal{C}}}{\cat{\SSet}}$ and a natural transformation $\eta : \homs \hoto \hathoms$ such that each $\eta_B : \homs{B} \to \hathoms{B}$ is a weak equivalence in the localised model structure and each $\hathoms{B}$ is fibrant in the localised model structure. 

\item If $\mathcal{C}$ is a model category, then the functor $\op{\Ho \mathcal{C}} \times \Ho \mathcal{C} \to \Ho \cat{\SSet}$ defined by $\tuple{A, B} \mapsto \hathoms{B} \argp{A}$ is (the functor part of) a derived hom-space functor $\RHom[\mathcal{C}] : \op{\Ho \mathcal{C}} \times \Ho \mathcal{C} \to \Ho \cat{\SSet}$.
\end{enumerate}
\end{thm}
\begin{proof}
(i). Use functorial fibrant replacements for the localised model structure on $\Func{\op{\mathcal{C}}}{\cat{\SSet}}$.

\bigskip\noindent
(ii). In view of \autoref{thm:derived.hom.spaces.via.hammock.localisation} and \autoref{prop:right.approximations.for.representable.presheaves}, this is an immediate consequence of Proposition 2.2 in \citep{Dwyer-Kan:1980b}.
\end{proof}

\ifdraftdoc

\else
  \printbibliography
\fi

\end{document}